\documentclass[12pt,a4paper,oneside]{article}
\usepackage[centertags]{amsmath}
\usepackage{amssymb,mathtools,mathrsfs}
\usepackage{amsthm}
\usepackage{enumerate}
\usepackage{stmaryrd}
\usepackage{dsfont}
\usepackage{url}
\usepackage{graphicx}
\usepackage{color}
\usepackage{epsfig}
\usepackage{geometry}
\usepackage{hyperref}

\usepackage{authblk}

\definecolor{dg}{rgb}{0,0.5,0}		%% normal green

\usepackage{hyphenat}
\usepackage[english]{babel}
\usepackage[T1]{fontenc}

\theoremstyle{theorem}
\newtheorem{Thm}{Theorem}[section]
\newtheorem{Cor}[Thm]{Corollary}
\newtheorem{Lem}[Thm]{Lemma}
\newtheorem{Prop}[Thm]{Proposition}

\theoremstyle{definition}
\newtheorem{Def}[Thm]{Definition}
\newtheorem{Rem}[Thm]{Remark}
\newtheorem{Ex}[Thm]{Example}

\newcommand{\affilit}[2]{\affil[#1]{\small{\textit{#2}}}}

\newcommand{\sR}{{\mathbb R}}

\newcommand{\B}{\mathcal{B}}

\newcommand{\vc}{\vcentcolon =}

\renewcommand{\bar}{\overline}		%% better bar
\title{The max-type quasimetrics on probability simplices}

\author[1]{Micha{\l} Eckstein}
\author[2]{Tomasz Miller}
\author[1,3]{Karol \.Zyczkowski}
\affilit{1}{Institute of Theoretical Physics, Faculty of Physics, Astronomy and Applied Computer Science, 
Jagiellonian University, ul.\ {\L}ojasiewicza 11, 30–348 Krak\'ow, Poland}
\affilit{2}{Copernicus Center for Interdisciplinary Studies, Jagiellonian University, ul. Szczepa\'nska 1/5, 31-011 Krak\'ow, Poland}
\affilit{3}{Center for Theoretical Physics, Polish Academy of Science, Al. Lotnik\'ow 32/46, 02-668 Warszawa, Poland}

\begin{document}

\maketitle

\begin{abstract}
Quasimetric spaces form a natural framework to study distance problems with an inherent directional asymmetry. We introduce a simple novel class of quasimetrics on probability simplices, inspired by the Chebyshev distance. It is shown that such quasimetrics have expedient geometric properties --- they induce the Euclidean topology and a Finslerian infinitesimal structure, with which the probability simplices become geodesic spaces. Moreover, we prove that the broad family of the proposed quasimetrics are monotone under bistochastic maps.
\end{abstract}


MSC classes: 53B40, 	53B12, 51F99

\section{Quasimetric spaces}

In several practical problems, it is necessary to consider distances that are inherently directional. Simple examples include irreversible thermodynamic processes and the theory of training and learning. Estimating travel time while sailing with or against the wind or navigating through mountainous terrain provide an additional motivation. These problems require concepts from metric theory -- such as the triangle inequality, path concatenation, and Lipschitz control -- without imposing an artificial symmetry that the underlying phenomenon does not possess. This naturally leads to the idea of an `asymmetric metric', which was formalised in 1931 by W.A. Wilson~\cite{Wilson1931}.

\begin{Def}
\label{def_quasimetric}
A \emph{quasimetric} on a set $X$ is a map $D: X \times X \rightarrow \sR$ satisfying, for all $x,y,z \in X$,
\begin{enumerate}
\item[\textbf{Q1}] $D(x,y) \geq 0$ \quad (nonnegativity),
\item[\textbf{Q2}] $D(x,y) = 0 \ \Leftrightarrow \ x = y$ \quad (nondegeneracy),
\item[\textbf{Q3}] $D(x,z) \leq D(x,y) + D(y,z)$ \quad (triangle inequality).
\end{enumerate}
The pair $(X,D)$ is called a \emph{quasimetric space}\footnote{The quasimetric spaces have also been called `oriented metric spaces' \cite{Bodjanova1981,Reilly1984} and `spaces with weak metrics' \cite{Ribeiro1943}. The notion of `quasimetric' is also used in the literature in a different sense, where it stands for a symmetric nondegenerate nonnegative function satisfying a relaxed triangle inequality, see e.g. \cite{QuasiOther1,QuasiOther2}.}.
\end{Def}

Any quasimetric $D$ can be ``symmetrised'' into a bona fide metric that bounds it from above. More concretely, the map $(x,y) \mapsto \max\{D(x,y),D(y,x)\}$ as well as $(x,y) \mapsto \left( D(x,y)^r + D(y,x)^r \right)^{1/r}$ for any $r \in [1,\infty)$ are such metrics. Indeed, the above maps are nondegenerate by $\textbf{Q2}$ and manifestly symmetric, whereas the triangle inequality follows from $\textbf{Q3}$ and Minkowski's inequality.

Quasimetrics have multifarious uses and applications. They connect naturally with the concepts of asymmetric normed spaces \cite{Cobzas2013}, quasi-uniform spaces \cite{Fletcher1982} and Riemann--Finsler geometry \cite{BaoChernShen2000}. In the latter context, they were employed to grasp the unidirection of  time in relativistic physics \cite{Randers1941}. Other applications of quasimetrics include: directed shortest-path distances on digraphs \cite{BangJensenGutin2009,Hepworth2025}, minimum cost sequence of edit operations needed to change one string into another \cite{WagnerFischer1974,Navarro2001}, shortest urban transportation through one-way streets \cite{Melo2022} or subsumption relation in machine learning \cite{gutierrez2002}.

%Randers (1941) and Bao–Chern–Shen (2000) for geometry/physics; Paxson (1997/1999) and Melo et al. (2022) for networks/transport; Wagner–Fischer/Navarro for edit costs; and Bang‑Jensen \& Gutin for digraph shortest paths.

\bigskip

Let $\Delta_N \subset \sR^{N+1}$ denote the $N$-dimensional simplex, the points of which we shall interpret as probability distributions on an $(N+1)$-element space, i.e.,
\begin{align*}
\Delta_N \vc \left\{ P = (p_0,p_1,\ldots,p_N) \in [0,1]^{N+1} \ \big| \ \sum p_i = 1 \right\}.
\end{align*}

In this work we introduce and study the family of `\emph{max-type}' quasimetrics on $\Delta_N$ defined as
\begin{align}
\label{Dfdef}
D_f(P,Q) \vc \max_i \big( f(q_i) - f(p_i) \big),
\end{align}
where $f: [0,1] \rightarrow \sR$ is a continuous strictly increasing map. In fact, without loss of generality we can rescale $f$ and assume it is a continuous bijection onto $[0,1]$ (satisfying, in particular, $f(0) = 0$ and $f(1) = 1$).

Formula \eqref{Dfdef} can be seen as an asymmetric generalisation of the \emph{Chebyshev distance}
\begin{align}\label{Chebyshev}
d_\infty(P,Q) \vc \max_i | q_i - p_i |, % = \lim_{r \to \infty} \left( \sum_{i=1}^n |q_i-p_i|^r \right)^{1/r},
\end{align} 
which measures distance under the uniform norm, picking up the worst-case coordinate difference. The Chebyshev distance is commonly used in logistics \cite{Langevin2005}, image processing \cite{Deborah2015}, pattern recognition \cite{Rodrigues2018} and data clustering \cite{Levy2025}.

The max-type quasimetrics \eqref{Dfdef}, albeit very simple, turn out to have many desirable properties (some of which require some additional, but rather mild assumptions on $f$). Although they give rise to the standard Euclidean topology on $\Delta_N$ (Section \ref{sec:2}, where we also recall the rudiments of the theory of quasimetric spaces), they possess the Finslerian infinitesimal structure and, moreover, $(\Delta_N, D_f)$ is a geodesic space (Section \ref{sec:3}). What is more, $D_f$ turns out to be monotone with respect to bistochastic maps for a broad choice of $f$ (Section \ref{sec:4}). It makes the max-type quasimetrics a potentially valuable tool in all domains where Chebyshev distances are employed and the problems exhibit a natural asymmetry.

%Before we delve into rigorously stating and proving these results, let us prepare the ground by recalling the quasimetric counterparts of the basic notions of metric geometry.

\section{Basic properties}
\label{sec:2}

Recall that, due to the lack of symmetry quasimetric spaces are, in general, \emph{bitopological} \cite{Kelly1963}. There are two types of balls:
\begin{itemize}
\item the \emph{forward balls} $\B^+_x(r) \vc \{ y \in X \, | \, D(x,y) < r \}$,
\item the \emph{backward balls} $\B^-_x(r) \vc \{ y \in X \, | \, D(y,x) < r \}$,
\end{itemize}
along with their closed counterparts:
\begin{itemize}
\item the \emph{closed forward balls} $\bar{\B}^+_x(r) \vc \{ y \in X \, | \, D(x,y) \leq r \}$,
\item the \emph{closed backward balls} $\bar{\B}^-_x(r) \vc \{ y \in X \, | \, D(y,x) \leq r \}$.
\end{itemize}
The family of forward (resp. backward) balls provides a base of the so-called \emph{forward topology} (resp. \emph{backward topology}). In general, these two topologies are different, and consequently one has to distinguish between forward/backward convergence, forward/backward continuity, forward/backward compactness and so on \cite{quasi19}. Needless to say, this is very cumbersome and probably partially explains why quasimetrics are not so commonly used despite their naturalness. Luckily however, in case of the quasimetrics given by (\ref{Dfdef}), the forward and backward topologies coincide, as we shall see below. This motivates adopting the following additional axiom.
\begin{Def}
\label{unitop}
The quasimetric space $(X,D)$ is called \emph{unitopological} if the forward and backward topologies coincide, what can be expressed by the condition\footnote{Note that any quasimetric space is first-countable with respect to both forward and backward topologies, and so it is sufficient to use sequences and not nets when discussing various topological notions.}
\begin{enumerate}
\item[\textbf{Q4}] $\forall x \in X \ \forall (x_n) \subset X \quad D(x,x_n) \rightarrow 0 \ \Leftrightarrow \ D(x_n,x) \rightarrow 0$.
\end{enumerate}
\end{Def}

\begin{Prop}
\label{prop1}
$(\Delta_N, D_f)$ defined above is a unitopological quasimetric space. Its topology is the standard Euclidean topology.
\end{Prop}
\begin{proof}
Let us first directly prove that formula (\ref{Dfdef}) indeed defines a quasimetric. To prove nonnegativity and nondegeneracy, assume that for some $P,Q \in \Delta_N$ we have $\max_i(f(q_i) - f(p_i)) \leq 0$, which means that $f(q_i) \leq f(p_i)$ for all $i$. Since $f$ is a continuous increasing bijection, the same concerns its inverse $f^{-1}$ and hence we obtain $p_i - q_i \geq 0$ for all $i$. But we know that $\sum_i(p_i - q_i) = 0$, and therefore $p_i = q_i$ for all $i$. We thus conclude that $\max_i(f(q_i) - f(p_i))$ cannot be negative and it vanishes if and only if $P = Q$.

As for the triangle inequality, the proof is trivial. Indeed, for any $P,Q,R \in \Delta_N$
\begin{align*}
D_f(P,Q) & = \max_i(f(q_i) - f(p_i)) = \max_i(f(r_i) - f(p_i) + f(q_i) - f(r_i))
\\
& \leq \max_i(f(r_i) - f(p_i)) + \max_i(f(q_i) - f(r_i)) = D_f(P,R) + D_f(R,Q).
\end{align*}
This finishes the proof that $D_f$ is a quasimetric.
\medskip

Let now $d_E$ denote the Euclidean metric on the simplex  $\Delta_N$. It is not hard to see that $d_E(P_n,P) \rightarrow 0$ implies both $D_f(P,P_n) \rightarrow 0$ and $D_f(P_n,P) \rightarrow 0$. For example, one might simply estimate $D_f(P,P_n)$ and $D_f(P_n,P)$ from above by $\max_i|f(p_{n,i}) - f(p_i)|$ and then invoke the continuity of $f$. Hence the Euclidean topology is finer than the forward and backward topologies. 

In order to realise that it is also coarser, it is instructive to see that the forward and backward balls are given by
\begin{align}
\label{balls}
& \B^+_P(r) \vc \{ Q \in \Delta_N \ | \ \forall \, i \quad f(q_i) < f(p_i) + r \},
\\
\nonumber
& \B^-_P(r) \vc \{ Q \in \Delta_N \ | \ \forall \, i \quad f(q_i) > f(p_i) - r \},
\end{align}
To see what this means geometrically, for any fixed $P$ let us introduce the maps $r \mapsto p_i^\pm(r)$ by
\begin{align}
\label{ppm}
& p_i^+(r) \vc f^{-1}\left(\min\{f(p_i) + r, 1\}\right),
\\
\nonumber
& p_i^-(r) \vc f^{-1}\left(\max\{f(p_i) - r, 0\}\right),
\end{align}
which are continuous by the continuity of $f$ and $f^{-1}$, and therefore
\begin{align}
\label{clim}
\forall \, i \quad \lim_{r \rightarrow 0^+} p_i^\pm(r) = p_i.
\end{align}

 Eq. (\ref{balls}) can be rewritten  with help of symbols $p_i^\pm$,
\begin{align}
\label{balls2}
& \B^+_P(r) \vc \{ Q \in \Delta_N \ | \ \forall \, i \quad q_i < p_i^+(r) \} = \Delta_N \cap \prod_{i=0}^N \, (-\infty,p_i^+(r)),
\\
\nonumber
& \B^-_P(r) \vc \{ Q \in \Delta_N \ | \ \forall \, i \quad q_i > p_i^-(r) \} = \Delta_N \cap \prod_{i=0}^N \, (p_i^-(r),+\infty).
\end{align}
In other words, $\B^+_P(r)$ is the intersection of $\Delta_N$ with the \emph{negative} orthant whose vertex has been translated to the point $P^+(r) \vc (p_0^+(r),\ldots,p_N^+(r))$. Similarly, $\B^-_P(r)$ is an intersection of $\Delta_N$ with the \emph{positive} orthant, the vertex of which has been translated to $P^-(r) \vc (p_0^-(r),\ldots,p_N^-(r))$ -- see Fig. \ref{fig2}. 

In order to prove that the Euclidean topology is coarser than the forward and backward topologies, it suffices now to take any Euclidean ball $B \subset \sR^{N+1}$ and for any point $P \in B \cap \Delta_N$ find $r > 0$ such that $\B^\pm_P(r) \subset B \cap \Delta_N$. To this end, first take another Euclidean ball $\B^E_P(\rho) \subset B$ centered at $P$ and with the radius $\rho$ sufficiently small to ensure the inclusion. Then, let $H$ be an open $(N+1)$-dimensional cube inscribed in $\B^E_P(\rho)$ (equivalently, take $H \vc \B^\infty_P(\rho/\sqrt{N+1})$, that is a $d_\infty$-ball with a suitably smaller radius). In light of (\ref{balls2}) and the subsequent discussion, one now simply has to take $r$ small enough so as to ensure that $P^\pm(r) \in H$. But that can always be done on the strength of (\ref{clim}).
\end{proof}

\begin{figure}[h]
\begin{center}
\resizebox{240pt}{!}{\includegraphics[scale=1.2]{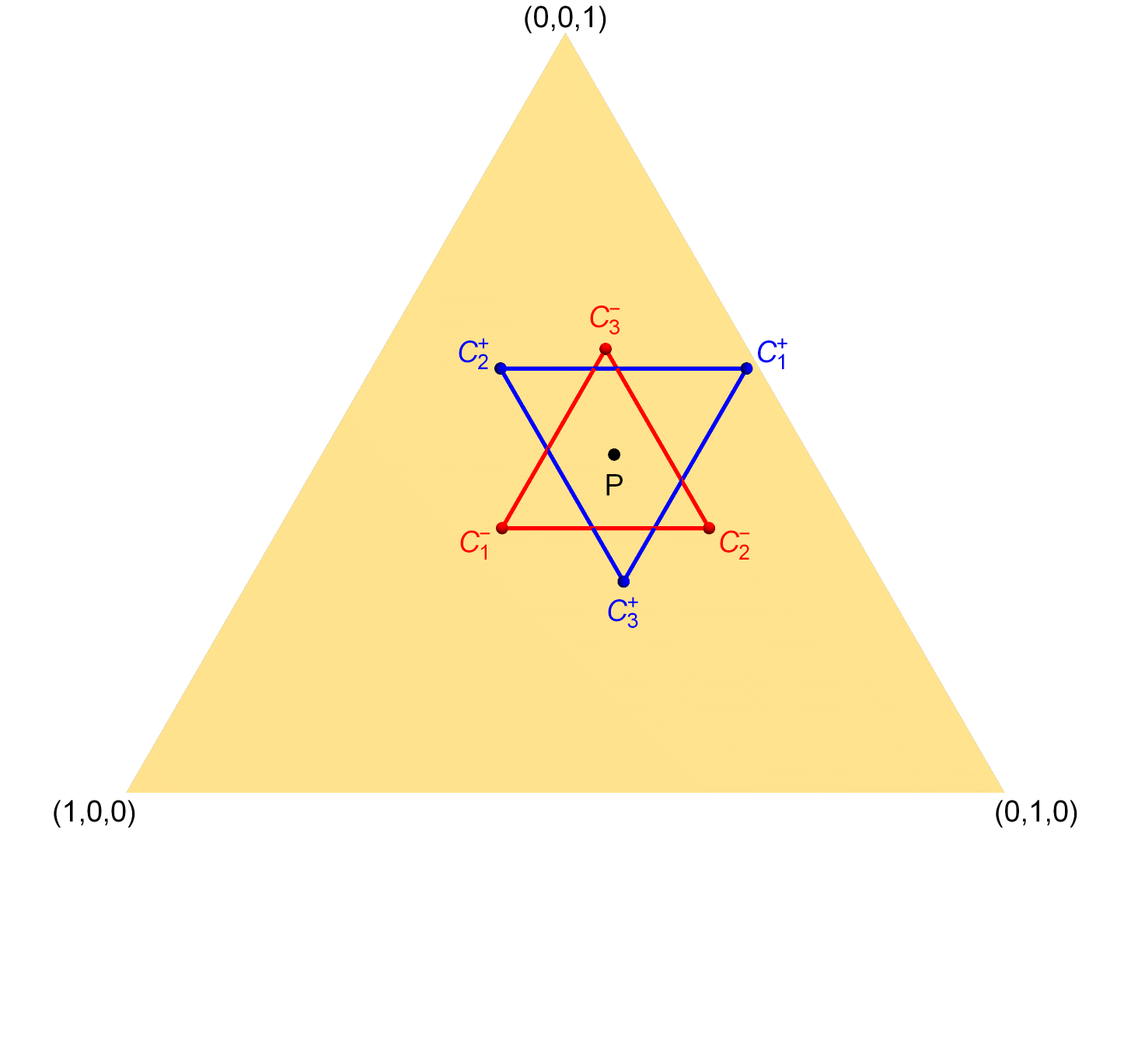}}
\caption{\label{fig1}The boundaries of certain forward (blue) and backward (red) quasimetric balls centered at $P = (2/9,1/3,4/9) \in \Delta_2$. Here $f(x) \vc x^{1/3}$, what explains why $P$ seems off-center and different apparent sizes of the balls. For the definition of $C^\pm_i$ see Remark \ref{rem_balls}.}
\end{center}
\end{figure}

\begin{figure}[h]
\begin{center}
\resizebox{240pt}{!}{\includegraphics[scale=1.2]{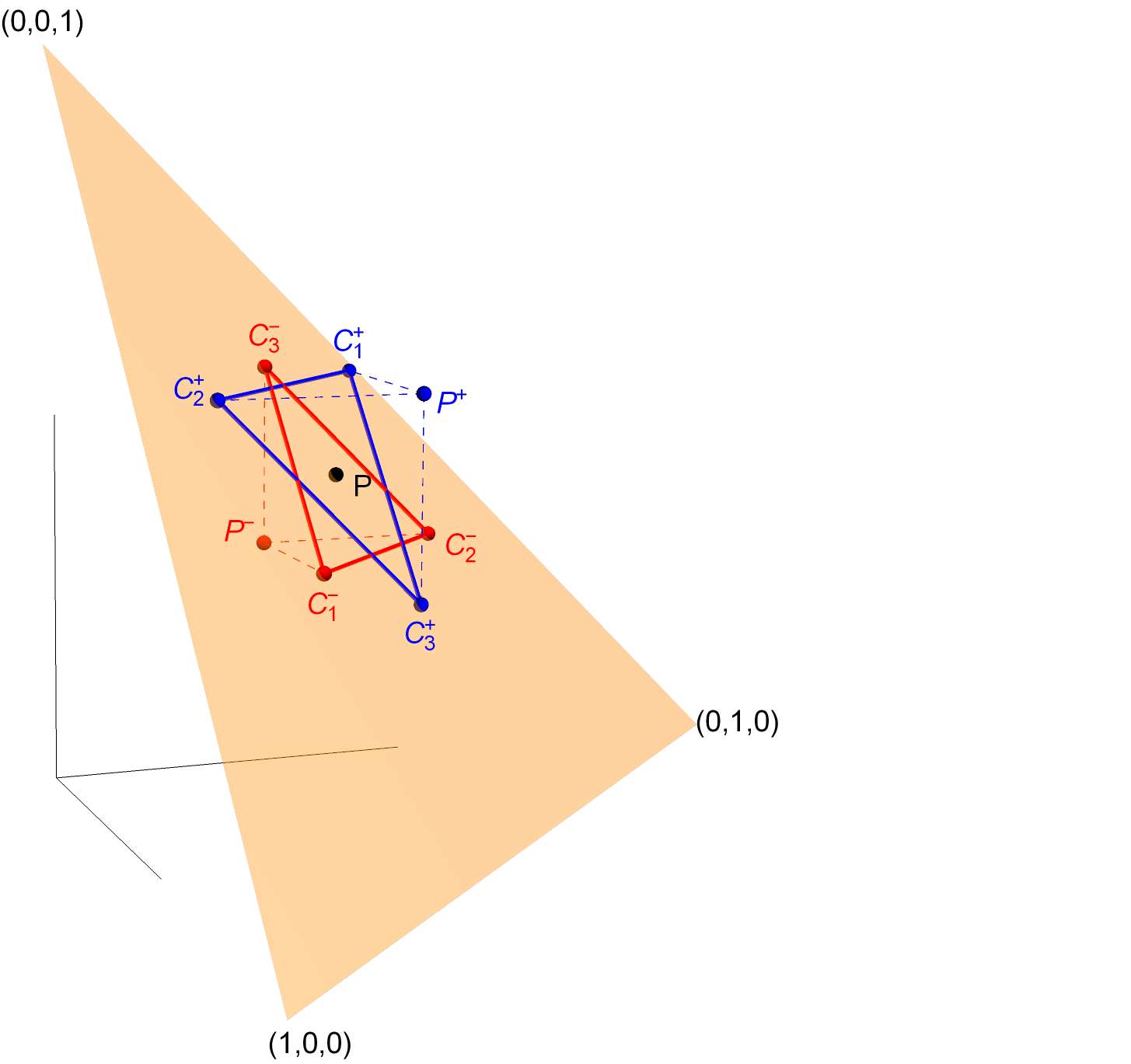}}
\caption{\label{fig2}The same (boundaries of) quasimetric balls as in Figure \ref{fig2}, only seen from the side, with the points $P^\pm \in \sR^3 \setminus \Delta_2$ visible.}
\end{center}
\end{figure}

\begin{Rem}
\label{rem_balls}
By (\ref{balls2}) and the subsequent discussion, it is not difficult to convince oneself that $\B^\pm_P(r)$ is an intersection of $\Delta_N$ with the interior of another $N$-simplex -- the convex closure of the set of points $C^\pm_0(r),C^\pm_1(r),\ldots,C^\pm_N(r)$ defined by
\begin{align*}
C^\pm_i \vc P^\pm + \left(1 - \sum\nolimits_j p_j^\pm\right)e_i,
\end{align*}
where $e_i$ is the $i$-th versor of $\sR^{N+1}$ and where we have suppressed the argument $r$ for legibility. These simplices' faces are parallel to the corresponding faces of $\Delta_N$. In fact, every backward ball $\B^-_P(r)$ is just a ``scaled-down'' version of $\Delta_N$, whereas every forward $\B^+_P(r)$ arises as an ``upside-down-scaled-down'' version of $\Delta_N$ (intersected with the latter). See Figures \ref{fig1} and \ref{fig2} for the illustration of the $\Delta_2$ case.
\end{Rem}

\begin{Ex}
Consider the simplest case when $f = \textnormal{id}$ and so $D_\textnormal{id}(P,Q) = \max_i(q_i - p_i)$. Then it is not difficult to demonstrate
 that
\begin{align*}
\frac{1}{N} d_\infty(P,Q) \leq D_\textnormal{id}(P,Q) \leq d_\infty(P,Q)
\end{align*}
with the left inequality saturated only for $Q = (\tfrac{1}{N+1},\ldots,\tfrac{1}{N+1})$ and $P$ a vertex. Such a quasimetric $D_f$ is thus bi-Lipschitz with respect to the Chebyshev distance (and thus to the Euclidean distance, too).

More generally, assume $f$ is locally bi-Lipschitz on $(0,1)$, i.e. if for any compact subinterval $[\alpha,\beta] \subset (0,1)$ there exists $L \geq 1$ such that $1/L|x-y| \leq |f(x) - f(y)| \leq L|x-y|$ for all $x,y \in [\alpha,\beta]$ (equivalently, assume that both $f$ and $f^{-1}$ are locally Lipschitz). Then for any $P \in \Delta^\circ_N$ and any $U \ni Q$ compactly supported in $\Delta^\circ_N$, one can find $L \geq 1$ such that
\begin{align*}
\frac{1}{L N} d_\infty(P,Q) \leq D_f(P,Q) \leq L d_\infty(P,Q)
\end{align*}
for any $Q \in U$.
\end{Ex}

Topologically, we are thus simply dealing with an Euclidean simplex. On the level of \emph{geometry}, however, the situation is more interesting, because (some of) the geometrical notions still enjoy the forward/backward distinction.

For example, one speaks about forward and backward Cauchy sequences as well as, consequently, about the forward and backward completeness -- see \cite[Definition 2.4]{quasi19} for the definitions.
\begin{Prop}
Quasimetric space $(\Delta_N,D_f)$ is both forward and backward complete.
\end{Prop}
\begin{proof}
Similarly as in metric spaces, in quasimetric spaces forward (resp. backward) compactness implies forward (resp. backward) completeness \cite[Theorem 2.1]{quasi19}. And $\Delta_N$ \emph{is} (forward and backward) compact, because it is so in the Euclidean sense.
\end{proof}

The above Proposition implies that the notions of forward and backward Cauchyness of sequences coincide in our space. More interestingly, however, the notions of forward/backward lengths of continuous curves do \emph{not} coincide and, consequently, the same concerns the forward/backward geodesics. In what follows, let $(X,D)$ be any unitopological quasimetric space.
\begin{Def}
\label{def_Lf}
The \emph{forward length} (f-length) of a curve $\gamma: [a,b] \rightarrow X$ is
\begin{align}
\label{Lf}
L^+(\gamma) := \sup\limits_{\pi} \sum_{i=1}^n D(\gamma(t_{i-1}), \gamma(t_{i})),
\end{align}
where the supremum runs over all partitions $\pi = \{t_0,t_1,\ldots,t_n\}, a = t_0 < t_1 < \ldots < t_n = b$ of the interval $[a,b]$. Equivalently, one can express (\ref{Lf}) as a limit
\begin{align}
\label{Lf1}
L^+(\gamma) := \lim\limits_{|\pi| \rightarrow 0} \sum_{i=1}^n D(\gamma(t_{i-1}), \gamma(t_{i})),
\end{align}
where $|\pi| := \max_i (t_i - t_{i-1})$ is the \emph{modulus} (a.k.a. \emph{mesh}) of the partition $\pi$.

The \emph{backward length} (b-length) $L^-(\gamma)$ is defined analogously -- one simply flips the arguments of $D$ in formulas (\ref{Lf}) and (\ref{Lf1}). If $L^+(\gamma)$ (resp. $L^-(\gamma)$) is finite, we call $\gamma$ \emph{f-rectifiable} (resp. \emph{b-rectifiable}).
\end{Def}

The equality between (\ref{Lf}) and (\ref{Lf1}) is nontrivial. The proof can be found in \cite[Proposition 1.1.10]{papa} for the metric case, but it remains valid also for quasimetrics after few obvious modifications. In fact, many other basic properties of length remain valid for f-length and b-length with their proofs almost unchanged -- modifications due to the lack of symmetry turn out to be minor.
\begin{Prop}
\label{prop_reparam}
Fix a curve $\gamma: [a,b] \rightarrow X$ and let $\rho: [a',b'] \rightarrow [a,b]$ be a monotonous surjection (which is automatically continuous) serving as a reparametrization map. Then for $\tilde{\gamma} \vc \gamma \circ \rho: [a',b'] \rightarrow X$ we have
\begin{itemize}
\item $L^\pm(\tilde{\gamma}) = L^\pm(\gamma)$ if $\rho$ is nondecreasing,
\item $L^\pm(\tilde{\gamma}) = L^\mp(\gamma)$ if $\rho$ is nonincreasing.
\end{itemize}
\end{Prop}
\begin{proof}
Cf. \cite[Proposition 1.1.8]{papa}.
\end{proof}

\begin{Prop}
The f-length and b-length are additive, meaning that $L^\pm(\gamma_1 \sqcup \gamma_2) = L^\pm(\gamma_1) + L^\pm(\gamma_2)$, where $\sqcup$ denotes the concatenation operation.
\end{Prop}
\begin{proof}
Cf. \cite[Proposition 1.1.12]{papa}.
\end{proof}

\begin{Prop}
For every f-rectifiable (resp. b-rectifiable) curve $\gamma: [a,b] \rightarrow X$, the map $t \mapsto L^+(\gamma|_{[a,t]})$ (resp. $t \mapsto L^-(\gamma|_{[a,t]})$) is increasing and continuous.
\end{Prop}
\begin{proof}
Cf. \cite[Proposition 1.1.13]{papa}.
\end{proof}

The above fact together with Proposition \ref{prop_reparam} imply that every curve can be (re)parametrised by f-length and b-length (cf. \cite[Section 1.2]{papa}).
\begin{Def}
An f-rectifiable (resp. b-rectifiable) curve $\gamma: [a,b] \rightarrow X$ is \emph{parametrised by f-length} (resp. \emph{b-length}) if for all $s,t \in [a,b]$, $s \leq t$, we have $L^+(\gamma|_{[s,t]}) = t - s$ (resp. $L^-(\gamma|_{[s,t]}) = t - s$).
\end{Def}

By the triangle inequality, for any curve $\gamma: [a,b] \rightarrow M$ one has $L^+(\gamma) \geq D(\gamma(a),\gamma(b))$ as well as $L^-(\gamma) \geq D(\gamma(b),\gamma(a))$. Geodesics are nothing but curves parametrised by length that saturate the above inequalities. Usually, however, one defines them as below and only later provides the above-mentioned minimising property (cf. \cite[Section 2.2]{papa}).

\begin{Def}
Curve $\gamma: [a,b] \rightarrow X$ is a \emph{forward geodesic} (f-geodesic) if
\begin{align*}
\forall \, s,t \in [a,b] \quad s \leq t \ \Rightarrow \ D(\gamma(s),\gamma(t)) = t - s.
\end{align*}
Similarly, it is a \emph{backward geodesic} (b-geodesic) if
\begin{align*}
\forall \, s,t \in [a,b] \quad s \leq t \ \Rightarrow \ D(\gamma(t),\gamma(s)) = t - s.
\end{align*}
\end{Def}

\begin{Prop}
\label{geodesics_char}
Any f-geodesic is parametrised by f-length. On the other hand, let $\gamma: [a,b] \rightarrow X$ be a curve parametrised by f-length. The following conditions are equivalent:
\begin{enumerate}
\item[(i)] $\gamma$ is an f-geodesic.
\item[(ii)] $\forall s,t \in [a,b], \ s \leq t \qquad D(\gamma(a),\gamma(t)) = D(\gamma(a),\gamma(s)) + D(\gamma(s),\gamma(t))$.
\item[(iii)] $L^+(\gamma) = D(\gamma(a),\gamma(b))$.
\end{enumerate}
\end{Prop}
\begin{proof}
$(i) \Rightarrow (ii)$ Take any $s,t \in [a,b]$, $s \leq t$ and observe that
\begin{align*}
D(\gamma(a),\gamma(t)) = t - a = s - a + t - s = D(\gamma(a),\gamma(s)) + D(\gamma(s),\gamma(t)).
\end{align*}

$(ii) \Rightarrow (iii)$ For any partition $\{t_i\}$ of $[a,b]$, on the strength of (ii) we obtain
\begin{align*}
\sum D(\gamma(t_{i-1}), \gamma(t_{i})) = D(\gamma(a),\gamma(b)).
\end{align*}
Taking the supremum over all partitions, we obtain (iii).

$(iii) \Rightarrow (i)$ Take any $s,t \in [a,b]$, $s \leq t$ and write (using the triangle inequality several times)
\begin{align*}
L^+(\gamma) = D(\gamma(a),\gamma(b)) & \leq D(\gamma(a),\gamma(s)) + D(\gamma(s),\gamma(t)) + D(\gamma(t),\gamma(b))
\\
& \leq D(\gamma(a),\gamma(s)) + L^+(\gamma|_{[s,t]}) + D(\gamma(t),\gamma(b))
\\
& \leq L^+(\gamma|_{[a,s]}) + L^+(\gamma|_{[s,t]}) + L^+(\gamma|_{[t,b]}) = L^+(\gamma),
\end{align*}
which means that all the above inequalities are in fact equalities. Hence $D(\gamma(s),\gamma(t)) = L^+(\gamma|_{[s,t]})$ and, since $\gamma$ is parametrised by f-length, the latter is equal to $t-s$. By the arbitrariness of $s,t$, this shows that $\gamma$ is an f-geodesic.
\end{proof}

\begin{Rem}
\label{rem_geodesics}
If the curve $\gamma: [a,b] \rightarrow X$ satisfies conditions (ii) or (iii) from the above proposition, then it becomes an f-geodesic after a suitable reparametrization. As such, it deserves the name \emph{f-pregeodesic}. The reparametrization map $\rho: [0, L^+(\gamma)] \rightarrow [a,b]$ is the inverse of the map $t \mapsto L^+(\gamma|_{[a,t]})$. The latter in this case can also be expressed as $t \mapsto D(\gamma(a),\gamma(t))$.

Needless to say, the above proposition has also its ``backward'' counterpart, featuring b-geodesics and b-length.
\end{Rem}

\section{Finslerian structure}
\label{sec:3}

In this section we show that $\Delta^\circ_N$, i.e. the interior of $\Delta_N$, possesses the structure of a Finsler manifold associated with the quasimetric $D_f$, provided $f \in C^1(0,1)$. We will see that --- under the additional condition that $f'$ is positive --- $D_f$ in fact \emph{arises} as a Finsler quasidistance.

First of all, let us recall some definitions and facts. In what follows, $X$ is a smooth manifold and $F: TX \rightarrow \sR_{\geq 0}$ is a Finsler function, i.e., a continuous function satisfying
\begin{itemize}
\item $\forall x \in X \ \forall v,w \in T_xX \quad F(v+w) \leq F(v) + F(w)$ (subadditivity)
\item $\forall \lambda > 0 \ \forall v \in TX \quad F(\lambda v) = \lambda F(v)$ (positive homogeneity)
\item $\forall v \in TX \quad F(v) = 0 \ \Rightarrow \ v = 0$ (nondegeneracy)
\end{itemize}
Note that we do \emph{not} assume that $F$ is smooth.

\begin{Def}
For any piecewise $C^1$ curve $\gamma: [a,b] \rightarrow X$ its \emph{Finsler length} is
\begin{align}
\label{ffinslerlength}
L_F(\gamma) := \int_a^b F(\dot{\gamma}(\tau))d\tau.
\end{align}
If the above integral is finite, $\gamma$ is called \emph{Finsler rectifiable}. Finsler length is invariant under positive reparametrization, and a Finsler rectifiable curve can always be \emph{(re)parametrised by Finsler length}, i.e., so as to satisfy $L_F(\gamma|_{[s,t]}) = t - s$ for any $s,t \in [a,b]$, $s < t$.

The \emph{Finsler quasidistance} between any $x,y \in X$ is
\begin{align}
\label{finslerdist}
D_F(x,y) := \inf \{L_F(\gamma) \ | \ \gamma \textrm{ -- piecewise } C^1 \textrm{ curve connecting } x \textrm{ with } y \}.
\end{align}

Finally, $\gamma$ is a \emph{Finsler geodesic} if it satisfies
\begin{align}
\label{finslergeodesic}
\forall \, s,t \in [a,b] \quad s \leq t \ \Rightarrow \ D_F(\gamma(s),\gamma(t)) = t - s.
\end{align}
\end{Def}

The pair $(X,D_F)$ is a unitopological quasimetric space (cf. \cite[(1.7)]{BM}), and the forward length associated to $D_F$ (via Definition \ref{def_Lf}) coincides with $L_F$ \cite[Theorem 4.2]{BM}\footnote{Bear in mind that the quasi-regularity requirement in \cite{BM} is automatically satisfied, because we \emph{define} $F$ to be subadditive, which is equivalent to the convexity of the indicatrix (cf. (2.1) in \cite{BM}), which in turn is equivalent to quasi-regularity (cf. paragraph just before Theorem 2.4 in \cite{BM}).}. This in particular means that Finsler geodesics enjoy the analogue of Proposition \ref{geodesics_char}. What is more, it therefore makes sense to talk about \emph{Finsler pregeodesics}, i.e. curves that satisfy conditions (ii) or (iii) and become geodesics after suitable reparametrization.

The following is a fundamental result for Finsler manifolds. 
\begin{Thm} \textup{\textbf{(Busemann--Mayer \cite[Theorem 4.1]{BM}\footnote{See the previous footnote. Cf. also \cite[Proposition 9.5.17]{SLK} for the smooth case.})}} The Finsler function $F: TX \rightarrow \sR_{\geq 0}$ can be retrieved from the Finsler quasidistance $D_F$ as the ``quasimetric derivative''
\begin{align}
\label{BM}
F(v) = \lim\limits_{t \rightarrow 0^+} \frac{D_F(\gamma(0),\gamma(t))}{t},
\end{align}
where $\gamma$ is any curve such that $\dot{\gamma}(0^+) = v$.
\end{Thm}

Limit (\ref{BM}) provides an excellent way to find the \emph{unique candidate} for the Finsler function giving rise to a given quasimetric. Calculating the quasimetric derivative of $D_f$, one finds the following candidate for $F$ living on $T\Delta^\circ_N$,
\begin{align}
\label{finsler2}
F(v) \vc \max\limits_{i} f'(p_i)v_i
\end{align}
for any $v \in T_P\Delta^\circ_N$, where $v_i$ is the $i$-th Cartesian component\footnote{Recall that we regard $\Delta_N$ as the subset of $\sR^{N+1}$. As a side note, notice that $\sum_i v_i = 0$.} of $v$. In order to prove that $F$ indeed gives rise to $D_f$ (i.e. that $D_f = D_F$), we begin with the following lemma.
\begin{Lem}
\label{lemFinsler}
Let $f \in C^1(0,1)$ and fix a $C^1$ curve $\gamma: [a,b] \rightarrow \Delta^\circ_N$. For any $\varepsilon > 0$ one can find $\delta > 0$ such that for any $s,t \in [a,b]$, $0 < t - s < \delta$,
\begin{align}
\label{finsler3}
\Bigg| \frac{D_f(\gamma(s),\gamma(t))}{t-s} - \underbrace{\max\limits_i (f \circ \gamma_i)'(s)}_{= \, F(\dot{\gamma}(s))} \Bigg| \leq \varepsilon.
\end{align}
\end{Lem}
\begin{proof} Begin by noticing that for any $i$, the map $t \mapsto (f \circ \gamma_i)'(t) = f'(\gamma_i(t))\dot{\gamma_i}(t)$ is uniformly continuous (being continuous on a compact interval $[a,b]$), i.e.,
\begin{align}
\label{keyfact}
\forall \varepsilon > 0 \ \ \exists \delta > 0 \ \ \forall s,t \in [a,b] \quad \ |t-s| < \delta \ \Rightarrow \ |(f \circ \gamma_i)'(t) - (f \circ \gamma_i)'(s)| \leq \varepsilon.
\end{align}
Let us rewrite it as
\begin{align*}
(f \circ \gamma_i)'(s) - \varepsilon \leq (f \circ \gamma_i)'(\tau) \leq (f \circ \gamma_i)'(s) + \varepsilon
\end{align*}
for any $\tau \in [s,t]$, where we assume that $s < t$. Applying now $\int_s^t d\tau$, then $\max_i$ and finally dividing everything by $(t-s)$, we obtain
\begin{align*}
\max\limits_i (f \circ \gamma_i)'(s) - \varepsilon \leq \frac{D_f\bigr(\gamma(s),\gamma(t)\bigl)}{t-s} \leq \max\limits_i (f \circ \gamma_i)'(s) + \varepsilon,
\end{align*}
which is precisely (\ref{finsler3}).
\end{proof}

The above result guarantees the equality between the forward and Finsler lengths of piecewise $C^1$ curves.%\footnote{In fact, even the curves which are only absolutely continuous wrt $d_E$ should be fine. One can naturally define f-AC curves and b-AC curves and show they are f-rectifiable and b-rectifiable, respectively. And $d_E$-AC probably implies f-AC and b-AC for the assumed $f$'s\ldots But this is for later considerations.}.

\begin{Thm}
\label{thm_finslerineq}
Let $f \in C^1(0,1)$ and $F$ be the Finsler function given by (\ref{finsler2}). Then for any piecewise $C^1$ curve $\gamma: [a,b] \rightarrow \Delta_N^\circ$ we have that
\begin{align}
\label{finslerineq}
& D_f\bigr(\gamma(a),\gamma(b)\bigl) \leq D_F\bigr(\gamma(a),\gamma(b)\bigl) \leq L_F(\gamma) = L^+(\gamma).
%\\
%& q(\gamma(b),\gamma(a)) \leq q_F(\gamma(b),\gamma(a)) \leq L^-_F(\gamma) = L^-(\gamma)
\end{align}
\end{Thm}
\begin{proof}
First we prove the equality in (\ref{finslerineq}). Let $\pi = \{a=t_0,t_1,\ldots,t_N = b\}$ be any partition of the interval $[a,b]$ respecting the way $\gamma$ is cut into $C^1$ pieces, i.e. we assume that $\gamma|_{[t_{i-1},t_{i}]}$ is $C^1$ for $i=1,\ldots,N$. On the strength of Lemma \ref{lemFinsler}, for any $\varepsilon > 0$, if only we ensure that $|\pi| < \delta$, we can write
\begin{align*}
& \left| \sum_{i=1}^N D_f\bigl(\gamma(t_{i-1}), \gamma(t_{i})\bigr) - \sum_{i=1}^N F\bigl(\dot{\gamma}(t_{i-1})\bigr)(t_i - t_{i-1}) \right|
\\
& \leq \sum_{i=1}^N (t_i - t_{i-1}) \left| \frac{D_f\bigl(\gamma(t_{i-1}), \gamma(t_{i})\bigr)}{t_i - t_{i-1}} - F\bigl(\dot{\gamma}(t_{i-1})\bigr) \right| \leq \varepsilon \sum_{i=1}^N (t_i - t_{i-1}) = \varepsilon (b-a).
\end{align*}
By (\ref{Lf1}) and the arbitrariness of $\varepsilon$, this means that
\begin{align*}
L^+(\gamma) = \lim\limits_{|\pi| \rightarrow 0} \sum_{i=1}^N D_f\bigl(\gamma(t_{i-1}), \gamma(t_{i})\bigr) = \lim\limits_{|\pi| \rightarrow 0} \sum_{i=1}^N  F\bigl(\dot{\gamma}(t_{i-1})\bigr)(t_i - t_{i-1}) = \int_a^b F\bigl(\dot{\gamma}(\tau)\bigr)d\tau = L_F(\gamma),
\end{align*}
provided one (and hence both) limits exist. The Riemann integral exists because $F$ given by (\ref{finsler2}) is continuous.

In order to obtain the leftmost inequality (the other one follows trivially from the very definition of Finsler distance (\ref{finslerdist})), notice that
\begin{align*}
D_f\bigl(\gamma(a), \gamma(b)\bigr) \leq L^+(\gamma) = L_F(\gamma)
\end{align*}
Taking the infimum over all piecewise $C^1$ curves going from $\gamma(a)$ to $\gamma(b)$, we obtain that $D_f\bigl(\gamma(a), \gamma(b)\bigr) \leq D_F\bigl(\gamma(a), \gamma(b)\bigr)$.
\end{proof}

\begin{Cor}
\label{cor_finslerineq}
A piecewise $C^1$ curve $\gamma: [a,b] \rightarrow \Delta_N^\circ$ is parametrised by f-length iff it is parametrised by Finsler length.  Moreover, if $\gamma$ is an f-(pre)geodesic, then it is also a Finsler (pre)geodesic.
\end{Cor}
\begin{proof}
As for the statement regarding pregeodesics, observe that if $\gamma$ is an f-pregeodesic, then the leftmost and the rightmost terms in (\ref{finslerineq}) are equal and therefore also $D_F\bigl(\gamma(a),\gamma(b)\bigr) = L_F(\gamma)$.

As for the remaining statements, apply (\ref{finslerineq}) to $\gamma|_{[s,t]}$ for any $s,t \in [a,b]$ such that $s < t$. Then $L^+(\gamma|_{[s,t]}) = L_F(\gamma|_{[s,t]})$ and so being parametrised by f-length is the same as being parametrised by Finsler length. Finally, concerning geodesics, one can simply invoke Proposition \ref{geodesics_char}, (iii) $\Rightarrow$ (i). More directly, one can also notice that if $\gamma$ is an f-geodesic then both the leftmost and the rightmost terms in (\ref{finslerineq}) applied to $\gamma|_{[s,t]}$ are equal to $t-s$. Therefore, one has also that $D_F\bigl(\gamma(s),\gamma(t)\bigr) = t - s$.
\end{proof}

%By Theorem \ref{midpoints} we know that \emph{every} pair of points $P,Q \in \Delta_N^\circ$ can be connected by an f-geodesic, and so it might seem that thanks to Theorem \ref{thm_finslerineq} we obtain that $D_f = D_F$ on $\Delta_N^\circ$. However, we are missing one crucial piece of information --- we don't know yet whether any of the connecting f-geodesics is piecewise $C^1$. Luckily quasimetric $D_f$ is simple enough that one can \emph{construct} such geodesics, significantly strengthening Theorem \ref{midpoints}.

Now comes the main theorem of this section. We show that \emph{every} pair of points $P,Q \in \Delta_N$ can be connected by an f-geodesic (and hence also by a b-geodesic). Moreover, under some regularity assumptions on $f$, these curves turn out to be regular as well.
\begin{Thm}
\label{geodesics}
The quasimetric space $(\Delta_N, D_f)$ is geodesic. What is more, if $f \in C^m(0,1)$ for $m = 1,2,\ldots$ with $f'$ positive, then any $P,Q \in \Delta_N^\circ$ can be connected by a $C^m$ f-geodesic.
\end{Thm}
\begin{proof}
Take any $P,Q \in \Delta_N$ and denote $r \vc D_f(P,Q)$. Define the map $[0,r] \ni t \mapsto \mu(t)$ implicitly via
\begin{align}
\label{geodesic2}
\sum_j f^{-1}\big(f(p_j) + t + \mu(t)(f(q_j) - f(p_j) - r)\big) = 1.
\end{align}
That $\mu$ is well defined can be deduced from the intermediate value theorem applied to the map $[0,1] \ni x \mapsto \sum_{j} f^{-1}\big(f(p_j) + t + x(f(q_j) - f(p_j) - r)\big) - 1$. What is more, the map $\mu$ is strictly increasing, because if there existed $0 \leq s < t \leq r$ such that $\mu(s) \geq \mu(t)$, we would obtain, by the fact that $f$ and $f^{-1}$ are strictly increasing and $f(q_j) - f(p_j) \leq r$ for every $j$, that
\begin{align*}
1 & = \sum_{j} f^{-1}\big(f(p_j) + s + \mu(s)(f(q_j) - f(p_j) - r)\big)
\\
& > \sum_{j} f^{-1}\big(f(p_j) + t + \mu(t)(f(q_j) - f(p_j) - r)\big) = 1,
\end{align*}
which is absurd. 

It is not hard to realise that $\mu$ is continuous and that $\mu(0) = 0$ and $\mu(r) = 1$. We are thus in power to define a curve $\gamma: [0,r] \rightarrow \Delta_N$ connecting $P$ with $Q$ as
\begin{align}
\label{geodesic3}
\gamma_j(t) \vc f^{-1}\big(f(p_j) + t + \mu(t)(f(q_j) - f(p_j) - r)\big)
\end{align}
for every $j$, with (\ref{geodesic2}) ensuring that $\gamma(t)$ is indeed in $\Delta_N$. Furthermore, observe that for any $0 \leq s < t \leq r$ we have
\begin{align*}
D_f\bigl(\gamma(s),\gamma(t)\bigr) & = \max_j \left( f(\gamma_j(t)) - f (\gamma_j(s)) \right) = \max_j \left[ t - s + (\mu(t) - \mu(s))(f(q_j) - f(p_j) - r) \right]
\\
& = t - s + (\mu(t) - \mu(s))(D_f(P,Q) - r) = t-s,
\end{align*}
where along the way we have used the fact that $\mu(s) < \mu(t)$. We have thus demonstrated that $\gamma$ is an f-geodesic.

Let now $P,Q \in \Delta_N^\circ$, which causes the image of $\gamma$ defined via (\ref{geodesic3}) to be contained in $\Delta_N^\circ$ as well. For any $m \geq 1$, if $f \in C^m(0,1)$ with $f'$ positive, then by the implicit function theorem also $f^{-1}$ and $\mu$ are $C^m$. But this means that also $\gamma$ is $C^m$ on the entire $[0,r]$.

As a final side remark, notice that applying $d/dt$ to (\ref{geodesic2}) leads to
\begin{align*}
\mu'(t) = \frac{\sum_j\frac{1}{f'(\gamma_j(t))}}{\sum_j \frac{f(q_j) - f(p_j) - r}{f'(\gamma_j(t))}},
\end{align*}
which is positive and continuous, as expected.
\end{proof}

\begin{Cor}
\label{rownosc}
Assume $f \in C^1(0,1)$ with $f'$ positive and $F$ is the Finsler function given by (\ref{finsler2}). Then $D_f = D_F$ on $\Delta_N^\circ$.
\end{Cor}
\begin{proof}
Take any $P,Q \in \Delta^\circ_N$. By Theorem \ref{geodesics}, there exists a $C^1$ f-geodesic $\gamma$ connecting them. But for such $\gamma$ inequalities in (\ref{finslerineq}) become equalities and hence in particular $D_f(P,Q) = D_F(P,Q)$.
\end{proof}

%\begin{Rem}
%There is an interesting geometrical picture behind the geodesics constructed in the proof of Theorem \ref{geodesics} [TO BE EXPANDED AND ILLUSTRATED\ldots]

%What is more, the construction carried over in the proof of Theorem \ref{geodesics} might be robust enough to transfer even to the \emph{infinite-dimensional} probability simplex, i.e. to the space of probability distributions on a countable set.
%\end{Rem}

The geodesics constructed in the proof of Theorem \ref{geodesics} usually are not the only geodesics connecting a given pair of points. In other words, the space $(\Delta_N, D_f)$ is not \emph{uniquely} geodesic, as illustrated by the following example.
\begin{Ex}
Let $Q = (1,0,\ldots,0)$. Then \emph{any} piecewise $C^1$ curve $\gamma$ connecting some $P \in \Delta_N$ with $Q$ is a pregeodesic, provided $\dot{\gamma}_i(t) < 0$ for $i=1,\ldots,N$ and for all $t$. If $P \in \Delta_N^\circ$, then there are infinitely many such curves. For example, consider the triangle $\Delta_2$ and all polygonal curves ending at the vertex $Q$, whose segments are at the angle no larger than $\pi/6$ from the altitude dropped at $Q$.
\end{Ex}

\section{Monotonicity}
\label{sec:4}

The natural mappings to consider on $\Delta_N$ are given by stochastic and bistochastic matrices. Recall that an $M$-by-$N$ matrix $S$ is \emph{stochastic} if $S_{ij} \geq 0$ and $\sum_{i=0}^M S_{ij} = 1$. Matrix $S$ is \emph{bistochastic} if it is a square stochastic matrix that additionally satisfies $\sum_{j=0}^N S_{ij} = 1$.

It is natural to ask how $D_f$ behaves under the action of (bi)stochastic matrices.

\begin{Thm}
Let $f \in C^1(0,1)$ with $1/f'$ positive and concave\footnote{Such functions include, e.g., power functions $x^\alpha$ with $\alpha \in (0,1]$, the log family $\tfrac{\ln(1+ax)}{\ln(1+a)}$ with $a>-1$ as well as the map $\tfrac{2}{\pi}\arcsin x$.}. Define $F: T\Delta^\circ_N \rightarrow \sR_{\geq 0}$ by (\ref{finsler2}), i.e. $F(v) \vc \max\limits_{i} f'(p_i)v_i$. Then
\begin{align}
\label{monotonicity1}
F(Sv) \leq F(v) 
\end{align}
for any bistochastic matrix $S$ and, as a consequence, for any $P,Q \in \Delta_N$ the monotonicity relation holds,
\begin{align}
\label{monotonicity2}
D_f(SP,SQ) \leq D_f(P,Q).
\end{align}
\end{Thm}
\begin{proof}
By the Birkhoff--von Neumann theorem, any bistochastic matrix $S$ is a convex combination of permutation matrices, $S = \sum_{j=1}^k \lambda_j \Pi_{\sigma_j}$, where $\Pi_{\sigma_j}$ is the matrix corresponding to the permutation $\sigma_j$. We thus have, for any $v$ tangent at $P \in \Delta^\circ_N$
\begin{align*}
F(Sv) = \max_i f'\Big( \sum_{j=1}^k \lambda_j p_{\sigma_j(i)} \Big) \sum_{j=1}^k \lambda_j v_{\sigma_j(i)}
\end{align*}
Observe now that for any sequences $(a_n),(b_n)$ with $a_n \geq 0$ and $b_n > 0$, one has that
\begin{align}
\label{ineq}
\frac{\sum_{j=1}^k a_j}{\sum_{j=1}^k b_j} \leq \max_{j=1,\ldots,k} \frac{a_j}{b_j}
\end{align}
for any $k$, simply because $\sum_j a_j = \sum_j \tfrac{a_j}{b_j} b_j \leq \max_j \tfrac{a_j}{b_j} \sum_{j'} b_{j'}$.  

Using the notation $w^+ \vc \max\{w,0\}$, we can now write, for any $i$,
\begin{align*}
& f'\Big( \sum_j \lambda_j p_{\sigma_j(i)} \Big) \sum_j \lambda_j v_{\sigma_j(i)} \leq f'\Big( \sum_j \lambda_j p_{\sigma_j(i)} \Big) \sum_j \lambda_j v^+_{\sigma_j(i)}
\\
& = \frac{\sum_j \lambda_j v^+_{\sigma_j(i)}}{\frac{1}{f'( \sum_j \lambda_j p_{\sigma_j(i)} )}} \leq \frac{\sum_j \lambda_j v^+_{\sigma_j(i)}}{\sum_j \lambda_j \frac{1}{f'(p_{\sigma_j(i)} )}} \leq \max_{j=1,\ldots,k} \frac{\lambda_j v^+_{\sigma_j(i)}}{\lambda_j \frac{1}{f'(p_{\sigma_j(i)} )}} = \max_{j=1,\ldots,k} f'(p_{\sigma_j(i)}) v^+_{\sigma_j(i)},
\end{align*}
where the two inequalities in the second line follow from the concavity of $1/f'$ and from (\ref{ineq}), respectively. Taking now $\max_i$ on the both sides we obtain
\begin{align*}
F(Sv) \leq \max_{i=0,\ldots,N} \max_{j=1,\ldots,k} f'(p_{\sigma_j(i)}) v^+_{\sigma_j(i)} = \max_i f'(p_i) v^+_i = \max_i f'(p_i) v_i = F(v),
\end{align*}
where the penultimate equality holds simply because all $v_i$'s cannot be negative (recall that $\sum_i v_i = 0$). This concludes the proof of (\ref{monotonicity1}).

In order to prove (\ref{monotonicity2}), take any $P,Q \in \Delta_N^\circ$ and any piecewise $C^1$ curve $\gamma: [a,b] \rightarrow \Delta_N^\circ$ connecting them. Then of course $S\gamma$ connects $SP$ with $SQ$ and by (\ref{monotonicity1}) we have
\begin{align*}
L_F(S\gamma) = \int_a^b F(S\dot{\gamma}(\tau)) d\tau \leq \int_a^b F(\dot{\gamma}(\tau)) d\tau = L_F(\gamma).
\end{align*}
Taking infimum over all $\gamma$'s of class $C^1$ connecting $P$ and $Q$, we obtain $D_F(SP,SQ) \leq D_F(P,Q)$, which is the same as $D_f(SP,SQ) \leq D_f(P,Q)$ on the strength of Corollary \ref{rownosc}. The latter extends onto entire $\Delta_N$ by continuity.
\end{proof}

\begin{Rem}
%Inequality 
Monotonicity relation (\ref{monotonicity2}) need not hold if matrix $S$ is stochastic but not bistochastic. Indeed, in the case $\Delta_2$ consider probability vectors $P = (1,0,0)$ and $Q = (0,\tfrac{1}{2},\tfrac{1}{2})$ along with the stochastic matrix
\begin{align*}
S = \left[ \begin{array}{ccc} 1 & 0 & 0 \\ 0 & 1 & 1 \\ 0 & 0 & 0 \end{array} \right].
\end{align*}
Then $SP = (1,0,0)$, $SQ = (0,1,0)$ and so 
\begin{align*}
D_f(SP,SQ) = 1 > f(\tfrac{1}{2}) = D_f(P,Q).
\end{align*}
\end{Rem}

\section*{Acknowledgements} 

ME and TM were supported by the National Science Centre in Poland under the research grant Sonata BIS (2023/50/E/ST2/00472).

\bibliographystyle{abbrv}
\bibliography{quasi}

\end{document}